\documentclass[11pt, oneside]{amsart}       % use "amsart" instead of "article" for AMSLaTeX format
\usepackage{geometry}                        % See geometry.pdf to learn the layout options. There are lots.
\usepackage[parfill]{parskip}            % Activate to begin paragraphs with an empty line rather than an indent
\usepackage{graphicx}                % Use pdf, png, jpg, or epsÂ§ with pdflatex; use eps in DVI mode
                            \usepackage{amsthm}    
\newtheorem{theorem}{Theorem}[section]
\newtheorem{definition}{Definition}[section]
\newtheorem{lemma}[theorem]{Lemma}
\newtheorem{corollary}[theorem]{Corollary}
\theoremstyle{remark}
\newtheorem*{remark}{Remark}
\usepackage{tcolorbox}
\newtheorem*{example}{Example}

\title{THE TOPOLOGY OF MEROMORPHIC LOOP SPACES OF SMOOTH CURVES}
\author{EMILE BOUAZIZ}
\begin{document}
\begin{abstract} We study the topology of the $\mathbb{C}$-points of the algebraic loop space of a smooth affine curve. In particular we will compute their sets of connected components. \end{abstract}
\maketitle
%\section{}
%\subsection{}
\section{acknowledgements}I thank Andrew MacPherson and Benjamin Hennion for numerous useful conversations. Further I thank Ian Grojnowski for shaping my entire outlook on Algebraic Geometry. Finally I thank MPIM Bonn for wonderful working conditions. \section{INTRODUCTION} If $M$ is, say, a manifold, then we can consider the free loop space, $LM=Map(S^{1},M)$, of unbased maps from a circle into $M$. If $M$ is connected then the set of connected components of $LM$ is identified with the fundamental group $\pi_{1}(M)$. There is an algebro-geometric version of $LM$, gotten heuristically by replacing the circle $S^{1}$ with a \emph{formal punctured disc}, $spec(\mathbb{C}((z)))$. The resulting object is an ind-scheme if we take our target to be affine. \\ \\The question of the topology of the $\mathbb{C}$-points of this ind-scheme now arises fairly naturally. We address this in the case of a smooth affine curve. To the best of our knowledge, the first non-trivial such computation is due to Contou-Carr\`{e}re, who studied the algebraic loop space of $\mathbb{G}_{m}$ in his work on what is now called the Contou-Carr\`{e}re symbol, (cf. [2]). In the case of $\mathbb{G}_{m}$ the answer turns out to \emph{agree} with the topological analogue. We will see that this is very much not the case in general. \section{BASICS} For the reader unfamiliar with algebraic loop spaces we include some background here. For more a more detailed introduction we refer the reader to the work of Kapranov and Vasserot  in [3]. \begin{definition} \begin{itemize}\item If $S=spec(A)$ is an affine scheme then we write $\mathcal{D}_{S}$ for the affine scheme $spec(A[[z]])$ and $\mathcal{D}_{S}^{*}$ for the affine scheme $spec(A((z)))$. We write $\mathcal{D}^{*}_{\mathbb{C}}$ for $\mathcal{D}^{*}_{spec(\mathbb{C})}$, and we sometimes shorten this further to just $\mathcal{D}^{*}$. \item If $X$ is a scheme then we define the pre-sheaves, $X(\mathcal{D})$ and $X(\mathcal{D}^{*})$, respectively by $X(\mathcal{D})(S)=X(\mathcal{D}_{S})$ and $X(\mathcal{D}^{*})=X(\mathcal{D}^{*}_{S}).$ We refer to them as the \emph{arc space} of $X$ and the \emph{loop space} of $X$ respectively. We will sometimes refer to them as \emph{holomorphic} (resp. \emph{meromorphic}) loops into $X$. The space $X(\mathcal{D})$ is a closed sub-space of $X(\mathcal{D}^{*})$. \item There is a \emph{covering action} of the monoid $\mathbb{N}^{\times}$ on the space $X(\mathcal{D}^{*})$ gotten by pre-composition with the etale covers, $z\mapsto z^{n}$. \end{itemize}\end{definition} We will briefly summarise the representability properties of these pre-sheaves: 
\begin{lemma}\begin{itemize}\item If $X$ is a scheme then the pre-sheaf $X(\mathcal{D})$ is representable by a scheme, affine if $X$ is.\item If $X$ is an affine scheme then  $X(\mathcal{D}^{*})$ is an ind-affine scheme, i.e. a filtered colimit (taken inside the pre-sheaf category) of affine schemes with transition maps closed inclusions.\end{itemize} \end{lemma} \begin{proof} This is well known, we refer the reader to \cite{KV}. One reduces to the case of $\mathbb{A}^{1}$ using compatability with limits and then deals with the $\mathbb{A}^{1}$ case explicitly. \end{proof}  We give a couple of examples of loop spaces of affine curves. \begin{example}\begin{itemize}\item We will briefly describe the loop space of $\mathbb{A}^{1}$. It can easily be seen to be the ind-scheme $\underrightarrow{\lim}_{n}\mathbb{A}^{[-n,\infty)}$. A Laurent power series $\gamma(z)=\sum_{i\geq -n}\gamma_{i}z^{i}$ corresponds to the element $(\gamma_{i})_{i}\in\mathbb{A}^{[-n,\infty)}$ in the obvious way. The sub-space of arcs is $\mathbb{A}^{[0,\infty)}$. Endowing the $\mathbb{C}$-points of this space with the ind-topology we note that it is contractible. We shall see below that, unsurprisingly, $\mathbb{A}^{1}$ is the only smooth curve for which this is the case. It is instructive to consider the following $\mathbb{A}^{1}$-family of loops into $\mathbb{A}^{1}$. Define $\gamma: \mathbb{A}^{1}\rightarrow\mathbb{A}^{1}(\mathcal{D}^{*})$ as follows: by the functor of points definition it corresponds to a map, $\mathcal{D}^{*}_{\mathbb{A}^{1}}\rightarrow\mathbb{A}^{1}$, i.e. to a function on the affine scheme $\mathcal{D}^{*}_{\mathbb{A}^{1}}$, and so to an element of the algebra $\mathbb{C}[t]((z))$. We take this element to be $z+tz^{-1}$. The fibre at $t=0$ is holomorphic but the other fibres have poles. We will see below that $\mathbb{A}^{1}$ is the only smooth curve for which this is possible. Indeed for every other smooth curve, $X$, the complex points of the arc space are \emph{open} in the complex points of the loop space.\item We consider now the case of $\mathbb{G}_{m}$. According to work of Contou-Carr\`{e}re, [2], if $A$ is a $\mathbb{C}$-algebra, then an invertible element of $A((z))$ admits a unique expression of the form, $$\alpha(z)=\alpha_{0}z^{\nu(\alpha)}\prod_{0<i<<\infty}(1-\alpha_{-i}z^{-i})\prod_{0<j}(1-\alpha_{j}z^{j}),$$ where $\nu(\alpha)$ is the order, $\alpha_{0}$ is a unit of $A$ and all the $\alpha_{-i}$ are nilpotents. This implies an isomorphism, $$\mathbb{G}_{m}(\mathcal{D}^{*})\cong\mathbb{Z}\times\mathcal{D}^{\infty}\times\mathbb{G}_{m}\times\mathbb{A}^{\infty}.$$ In particular we see that $\pi_{0}(\mathbb{G}_{m}(\mathcal{D}^{*}_{\mathbb{C}}))=\mathbb{Z}$, corresponding to pole/ zero order. We will find it useful, with a view to generalising this result, to decompose $\mathbb{Z}$ as $(\infty\times\mathbb{Z}_{<0})\sqcup\{0\}\sqcup (0\times\mathbb{Z}_{>0})$. Recall that we have a covering action of the monoid $\mathbb{N}^{\times}$ on $\mathbb{G}_{m}(\mathcal{D}^{*}_{\mathbb{C}})$. This descends to an action on $\pi_{0}$. It is easily seen that it acts trivially on the summand, $\{0\}$, of $\pi_{0}$ and via multiplication on the summands $(\infty\times\mathbb{Z}_{<0})$ and $(0\times\mathbb{Z}_{<0})$. Note that in this case the answer agrees with the topological analogue. \\ \\ Taking the quotient by this action, and denoting it $\pi_{0}^{\mathbb{N}^{\times}}$, we see that this has size $3$; one component for each puncture and one for the holomorphic loops. It is in this form that the computation will generalise. \item Already in higher genera it is no longer the case that the complex points of the algebraic loop space $X(\mathcal{D}^{*})$ are homotopy equivalent to their topological analogue, cf. the mathoverflow comment of B. Bhatt, [1]. \end{itemize}\end{example} We introduce some notation before stating our main result; \begin{definition} If $X$ is an affine smooth curve with proper model $\overline{X}$, then we let $\partial X$ denote the complement of $X$ inside $\overline{X}.$ \end{definition} \begin{tcolorbox}\begin{theorem} Let $X$ be a smooth complex affine curve, then we compute the number of connected components, up to coverings, of its loop space as $$\#\pi_{0}^{\mathbb{N}^{\times}}X(\mathcal{D}_{\mathbb{C}}^{*})=1+\#\partial X,$$ if $X$ is not $\mathbb{A}^{1}$, and $1$ otherwise. \end{theorem}\end{tcolorbox}\begin{remark} We can re-state our main result more explicitly as follows: let $X$ be a smooth affine curve. If $X=\mathbb{A}^{1}$ then its loop space is connected. Otherwise the set of connected components of its loop space is naturally identified with $(\mathbb{N}^{\times}\times\partial X)\bigsqcup \{X(\mathcal{D})\}$.\end{remark} \section{THE PROOF} We will break the proof into a couple of cases, namely those curves $X$ with $\#\partial X>1$, and those with $\#\partial X=1,$  the second case proving to tbe the trickier one. Our strategy will be to find invariants of loops which do not change as we continuously deform the loop, and then to use these to differentiate between components. \begin{definition} If $\gamma:\mathcal{D}^{*}_{\mathbb{C}}\rightarrow X$ is a loop, we denote by $\overline{\gamma}$ the extension to a holomorphic loop, $\overline{\gamma}:\mathcal{D}_{\mathbb{C}}\rightarrow\overline{X}$.\end{definition}\begin{remark} The existence of such an extension is guaranteed by the valuative criterion. We caution the reader that this \emph{does not} correspond to a morphism of spaces, $X(\mathcal{D}^{*})\rightarrow\overline{X}(\mathcal{D})$. This can be seen for example by considering the $\mathbb{A}^{1}$-family of loops, $z+tz^{-1}$, defined above, indeed the extensions to the central element $0\in\mathcal{D}$ vary discontinuosly with $t$. \end{remark} The most natural invariant associated to a loop is the pull-back on degree $1$ de Rham cohomology. In fact it makes more sense to consider \emph{continuous de Rham cohomology} of topological $\mathbb{C}$-algebras. We record a simple lemma: \begin{lemma} Let $S$ be a connected affine scheme with $\mathbb{C}$-points $s_{0}$ and $s_{1}$. Let $$\gamma:\mathcal{D}^{*}_{S}\rightarrow X,$$ be an $S$-family of loops. Write $\gamma_{0}$ and $\gamma_{1}$ for the fibres at  $s_{0}$ and $s_{1}$. Then the pull-backs, $$s_{i}^{*}:H^{*}_{dR}(X)\rightarrow H^{*}_{dR}(\mathcal{D}^{*}),$$ agree for $i=0,1$, where $H^{*}_{dR}$ denotes continuous de Rham cohomology.\end{lemma} \begin{proof} This is a simple generalisation of the corresponding statement for discrete de Rham cohomology. \end{proof} \begin{remark}We remind the reader that the first continuous de Rham cohomology of $\mathcal{D}^{*}_{\mathbb{C}}$ is one dimensional, generated by $d\log(z)$. Suppose we have a loop, $\gamma$, into $X$, with $\overline{\gamma}(0)=p\in\overline{X}$. Suppose further that we have a meromorphic $1$-form $\omega$ on $X$, then the pull-back in first continuous de Rham cohomology corresponds to a multiple (given by the order at $0$) of the residue.\end{remark} We show now that if two non-holomorphic loops admit extensions with different central value then they must lie in different components. In order to do this we first record a simple lemma: \begin{lemma} Let $p\neq q$ be two points of a smooth proper curve $X$. Then there exists a meromorphic $1$-form, $\omega$, on $X$, which is holomorphic away from $\{p,q\}$ and which satisfies $res_{p}(\omega)=1$, $res_{q}(\omega)=-1$.\end{lemma}\begin{proof} This is a simple computation with the Riemann-Roch Theorem. Let $\Omega_{X}$ dentoe the canonical bundle. We apply Riemann-Roch to the bundle $\Omega_{X}(p+q)$. We note $$deg(\Omega_{X}(p+q))=2g(X),$$ and thus we compute $h^{0}(\Omega_{X}(p+q))=1+g > h^{0}(\Omega_{X})$. Taking a non-zero element, $$\omega\in H^{0}(X,\Omega_{X}(p+q))/H^{0}(X,\Omega_{X}),$$ we may assume without loss of generality that it is not homolorphic at $p$. It has a pole of order at most one at $p$ and thus a pole of order exactly one. Rescaling we may assume $res_{p}(\omega)=1$, the only other point with potentially non-zero residue is $q$ whence we deduce  $res_{q}(\omega)=-1$ by the Residue Theorem, completing the proof. \end{proof} We deduce the following: \begin{corollary} Let $X$ be an affine smooth curve with $\#\partial X>1$. Let $\gamma_{0}$ and $\gamma_{1}$ be non-holomorphic loops into $X$ and let $\overline{\gamma_{0}}$ and $\overline{\gamma_{1}}$ be their extensions to holomorphic loops into $\overline{X}$. Assume that the values at $0\in\mathcal{D}$ of these extensions differ. Then $[\gamma_{0}]\neq [\gamma_{1}]$ in $\pi_{0}X(\mathcal{D}^{*}_{\mathbb{C}})$. Further, if $\gamma_{2}$ is a homolorphic loop into $X$ then its class in $\pi_{0}$ differs from either of the classes of the $\gamma_{0/1}$.\end{corollary} \begin{proof}. The lemma above furnishes a meromorphic $1$-form $\omega$, so that the three residues $res_{\overline{\gamma}_{i}(0)}(\omega)$ all differ. Recalling that these residues correspond to de Rham pull-backs and then combining the lemmas above it suffices to show that if two loops,  $\gamma_{0}$ and $\gamma_{1}$, are in the same connected component of $X(\mathcal{D}^{*}_{\mathbb{C}})$  then there is a connected affine $S$ and an $S$-family of loops interpolating them.\\ \\ Let the connected component containing them be called $S´$. It is ind-affine and thus we find an affine closed sub-scheme containing the loops  $\gamma_{0}$ and $\gamma_{1}$. They must lie in the same component of a suitably large such sub-scheme, which we call $S$. The inclusion $S\hookrightarrow X(\mathcal{D}^{*})$ provides the tautological $S$-family, $\gamma: \mathcal{D}^{*}_{S}\rightarrow X$ and we conclude. \end{proof}\begin{remark} The reader will note that this argument breaks down without the assumption that  $\#\partial X>1$. Indeed if  $\#\partial X=1$ then any holomorphic form, $\omega$, on $X$ will have trivial residue at the puncture of $\overline{X}$.\end{remark} The next step is to identify when two loops, $\gamma_{0}$ and $\gamma_{1}$, both extending to arcs with central value $p\in\overline{X}$, lie in the same connected component. There is an obvious guess, and it is the correct one. We first note that in a purely local situation, the computation of the connected components of a loop space is easy: \begin{lemma} Let $\mathfrak{X}=\mathcal{D}^{*}(\mathcal{D}^{*})$. Then connected components of $\mathfrak{X}(\mathbb{C})$ are labelled by $\mathbb{Z}_{>0}$, corresponding to the degree of the morphism $\mathcal{D}^{*}\rightarrow\mathcal{D}^{*}$. \end{lemma}\begin{proof} First note that if $\gamma:\mathcal{D}^{*}\rightarrow\mathcal{D}^{*}$ is a loop of order $n$ then $\gamma^{*}(d\log(z))=nd\log(z)$ so that loops of different degree must lie in different connected components according to lemm the lemma above. Writing a loop of order $n$ in the form $\gamma(z)=z^{n}\gamma_{+}(z)$ gives a homeomorphism with $\mathbb{C}^{\times}\times\mathbb{C}^{\infty}$ and thus we conclude the lemma. With more effort we could describe the space $\mathfrak{X}$ in the manner of Contou-Carr\`{e}re (for $\mathbb{G}_{m}$) but this is not needed for this note. \end{proof}\begin{corollary} It follows that after we quotient by the covering action we have just one component, i.e. we have  $\pi_{0}^{\mathbb{N}^{\times}}\mathfrak{X}=*.$\end{corollary}\begin{corollary}If two loops, $\gamma_{0}$ and $\gamma_{1}$, both extend to arcs having the same central value, $p\in\overline{X}$, then they lie in the same component of $X(\mathcal{D}^{*}_{\mathbb{C}})$ iff they have the same order at $p$.\end{corollary}\begin{proof} Write $\mathcal{D}^{*}_{p}\hookrightarrow X$ for the inclusion of the formal punctured neighbourhood of $p$. It suffices to note that our loops $\gamma_{i}$ admit unique factorisations through $\mathcal{D}^{*}_{p}$ and then to use the lemma above. \end{proof}\begin{remark} This suffices to prove the main theorem in the case of curves with at least two punctures once we note that the arc space is also connected. In summary we see that in this case we have one component corresponding to the arcs, and one for each pair consisting of a positive integer and a puncture of the curve $X$. After quotienting by the $\mathbb{N}^{\times}$-action, we arrive at the desired claim. \end{remark} \begin{remark}
Our methods above, dealing with the case of curves $X$ with $\#\partial X>1$, were all essentially abelian. Abelian invariants, say $H^{*}_{et}(-,\mathbb{Z}/l)$, will not distinguish $X$ from $\overline{X}$ in the case of $\#\partial X=1$. Indeed the set of etale $\mathbb{Z}/l$-torsors for $X$ and $\overline{X}$ agree as they have the same \emph{abelianised} etale fundamental groups. This suggests that we attempt to distinguish non-holomorphic loops from holomorphic ones via the induced map on $\pi_{1}^{et}$. Note that we know from above that there are \emph{at most} two components (up to covers), our goal below is to show that there are exactly two as soon as $g(\overline{X})>0$. We remind the reader that the Newton-Puiseaux theorem implies that we have $\pi_{1}^{et}(\mathcal{D}^{*})=\widehat{\mathbb{Z}}$. Further the Riemann Existence theorem implies that we have $\pi_{1}^{et}(X)=\widehat{\pi_{1}}(X(\mathbb{C}))$, the profinite completion of the topological fundamental group.\end{remark} We begin with the following simple lemma: \begin{lemma} Let $X$ be a smooth curve with $\#\partial X>1$, and $g(\overline{X})>0$, then there exists an etale cover, $Y\rightarrow X,$ which is not pulled back from an etale cover of $\overline{X}$, i.e. does not extend to an etale cover of $\overline{X}$.\end{lemma}\begin{proof} Let us take standard generators for $\pi_{1}(\overline{X}(\mathbb{C}))$ coming from viewing a genus $g$ real surface as a $4g$-gon with appropriate edge identifications. We will call the generators $\{\alpha_{1},\beta_{1},...,\alpha_{g},\beta_{g}\}$.  $\pi_{1}(\overline{X}(\mathbb{C}))$ is generated by these elements subject to the one relation $\prod_{i}[\alpha_{i},\beta_{i}]=1$. $\pi_{1}(X(\mathbb{C}))$ is freely generated by them. \\ \\ It suffices, by Riemann Existence, to find a finite group $G$ such that the natural map, $$Hom(\pi_{1}(\overline{X}(\mathbb{C})),G)\rightarrow Hom(\pi_{1}(X(\mathbb{C}),G),$$ is not surjective. We take $G=S_{3}$. We define an element of the right hand side by sending $\alpha_{1}\mapsto (12), \beta_{1}\mapsto (123)$ and the rest of the generators to the unit. This maps the element  $\prod_{i}[\alpha_{i},\beta_{i}]=1$ to a non-trivial permutation and thus we are done.\end{proof}\begin{corollary} Let $X$ be as above, and fix any non-holomorphic loop, $\gamma\in X(\mathcal{D}^{*}_{\mathbb{C}})$. Then the induced map of etale fundamental groups, $\widehat{\mathbb{Z}}=\pi_{1}^{et}(\mathcal{D}^{*})\rightarrow\pi_{1}^{et}(X)$, is non-zero. \end{corollary}\begin{proof} Let $\infty\in\overline{X}$ denote the puncture point. We may assume the extension, $\overline{\gamma}:\mathcal{D}\rightarrow\overline{X}$, maps $0$ to $\infty$ to order $1$, since otherwise the result is simply to pre-compose with the map induced on $\pi_{1}^{et}(\mathcal{D}^{*})$ by a degree $n$ cover, which is of course not zero. \\ \\ To show that the induced map,  $\widehat{\mathbb{Z}}=\pi_{1}^{et}(\mathcal{D}^{*})\rightarrow\pi_{1}^{et}(X)$, is non-zero it suffices to show that there is a non-trivial etale cover of $X$ which pulls back to a non-trivial etale cover of $\mathcal{D}^{*}$ under $\gamma$. \\ \\ However, let us note that if $\gamma$ pulls back a cover, $Y\rightarrow X$, to a trivial cover of $\mathcal{D}^{*}$, then the pulled-back cover extends over the puncture point $0\in\mathcal{D}$, whence the original cover extends over the puncture $\infty\in\overline{X}$. We have seen above that there exist covers not extendable in this manner, whence the proof is complete.  \end{proof} \begin{corollary} Let $\gamma_{0}$ be a holomorphic and $\gamma_{1}$ a be non-holomorphic element of $X(\mathcal{D}^{*}_{\mathbb{C}})$, for $X$ as above, then the classes of  $\gamma_{0}$ and $\gamma_{1}$ differ in $\pi_{0}(X(\mathcal{D}^{*}_{\mathbb{C}})).$ Further, two non-holomorphic loops extending to arcs with the same central value lie in the same component iff they have the same degree. \end{corollary}\begin{proof} Note that a holomorphic loop induces zero on $\pi_{1}^{et}$, further we can detect the order from the induced map on $\pi_{1}^{et}$. As such, using arguments similar to Lemma 4.1 above, it suffices to show the following: if $$\gamma:\mathcal{D}^{*}_{S}\rightarrow X,$$  is an $S$-family of loops, with $S$ connected, and with $\gamma_{0}$ and $\gamma_{1}$  the fibres at  $s_{0}$ and $s_{1}$, then maps induced on $\pi_{1}^{et}$ by the $\gamma_{i}$ agree. This follows from the results of SGA, specifically it can be deduced from the results of [SGA 1, Exp. X, Cor 2.4, Thm 3.8, Cor. 3.9 \& Section 2 of Exp. XIII], many thanks to Jason Starr for the detailed reference. \end{proof} Putting this all together we deduce the main theorem, which we re-state below:  \begin{tcolorbox}\begin{theorem}  let $X$ be a smooth affine curve. If $X=\mathbb{A}^{1}$ then its loop space is contractible. Otherwise the set of connected components of its loop space is naturally identified with $(\mathbb{N}^{\times}\times\partial X)\sqcup \{X(\mathcal{D})\}.$ In particular, the number of connected components of the loop space of such an $X$, up to coverings of the punctured disc, is $1$ plus the number of points needed to compactify $X$.\end{theorem}\end{tcolorbox}

\end{document}